\documentclass{amsart}[12pt]

  \usepackage[utf8]{inputenc}
  \usepackage{amsmath}
\usepackage{enumerate}
  \usepackage{amssymb}
  \usepackage{amsthm}
  \usepackage{bm}
  \usepackage{graphicx}
  \usepackage{color}
  \usepackage{accents}

\newtheorem{Th}{Theorem}[section]
\newtheorem{lem}[Th]{Lemma}

\theoremstyle{definition}
\newtheorem{Def}[Th]{Definition}

\newtheorem{Cor}[Th]{Corollary}
\newtheorem{Prop}[Th]{Proposition}

\theoremstyle{remark}
\newtheorem{rem}[Th]{Remark}

\newtheorem{thank}{\ \ \ Acknowledgment}

\def\scalar(#1,#2){(#1\mid#2)}

\newcommand{\ca}{\mathcal{A}}

\newcommand{\cb}{\mathcal{B}}
\newcommand{\cc}{\mathcal{C}}
\newcommand{\cu}{\mathcal{U}}

\newcommand{\ci}{\mathfrak{I}}

\newcommand{\Q}{\mathbb{Q}}

\newcommand{\Z}{{\mathbb{Z}}}

\newcommand{\N}{{\mathbb{N}}}
\newcommand{\E}{{\mathbb{E}}}
\newcommand{\1}{{\mathds{1}}}

\def\1{\,\rlap{\mbox{\small\rm 1}}\kern.15em 1}

\def\build#1_#2^#3{\mathrel{\mathop{\kern 0pt#1}\limits_{#2}^{#3}}}

\def\converge#1#2#3#4{\build\hbox to
#1mm{\rightarrowfill}_{#2\rightarrow #3}^{\hbox{\scriptsize #4}}}
\newcommand{\tend}[3][]{\xrightarrow[#2\to#3]{#1}}

\newcommand{\ds}{\displaystyle}

\title[On the pointwise convergence of multiple ergodic averages: G.S.]{On the pointwise convergence of multiple ergodic averages
and non-singular dynamical systems}
\author{E. H. el Abdalaoui}
\address{University of Rouen Normandy  \\
  LMRS  UMR 60 85 CNRS-univ, Departement of Mathematics,
Avenue de l'Universit\'e, BP.12
76801 Saint Etienne du Rouvray - France }
\email{elhoucein.elabdalaoui@univ-rouen.fr}

\thanks{}

\date{\today}
\begin{document}
\begin{abstract}It is shown that there is a non-singular dynamical system 
for which the maximal ergodic inequality does not hold. The proof is accomplished by proving that 
there exist a subsequence for which the multiple ergodic averages of commuting invertible
measure preserving transformations of a Lebesgue probability space converge almost everywhere provided
that the maps are weakly mixing with an ergodic extra condition. We further get that the non-singular strategy to solve the pointwise convergence of the Furstenberg ergodic averages fails.
\\
\end{abstract}

\maketitle
\section{Introduction}
The purpose of this note is to establish that there exist a non-singular dynamical system 
for which the maximal ergodic inequality does not hold. To this end, it is shown that the well-known open problem 
of the pointwise convergence of the Furstenberg ergodic averages has a positive answer if we restrict our self to the convergence along a subsequence. We remind that this problem can be formulated as follows:
let $k\geq 2$, $(X,\cb,\mu,T_i)_{i=1}^{k}$ be a finite family of dynamical systems where $\mu$ is a probability measure, $T_i$ are commuting  invertible  measure preserving transformations and  $f_1,f_2,\cdots,f_k$  a finite family of bounded functions. Does the following averages
$$\frac1{N}\sum_{n=1}^{N}\prod_{i=1}^{k}f_i(T_i^nx)$$
convergence almost everywhere?\\

The classical Birkhoff theorem correspond to the case $k=1$. The case $k=2$ with
$T_i=T^{p_i}$, and $p_i \in \N^*$ for each $i$, is covered by Bourgain double ergodic theorem \cite{BourgainD}.\\

 Here, our aim is to establish that there exists a subsequence $(N_l)$ such that for any  finite family of bounded functions $f_1,f_2,\cdots,f_k$  ,
$$\frac1{N_l}\sum_{n=1}^{N_l}\prod_{i=1}^{k}f_i(T_i^nx)$$
convergence almost everywhere.\\

We remind that the $L^2$ version of this problem has been intensively studied and the topics is nowadays very rich. These studies were originated in the seminal work of Furstenberg on Szemeredi's theorem in \cite{Furstenberg1}.  Furstenberg-Kaztnelson-Ornstein in \cite{FKO} proved that
the $L^2$-norm convergence holds for $T_i=T^i$ and $T$ weakly mixing. Twenty three years later, Host and Kra \cite{Host} and independently T. Ziegler \cite{Ziegler} extended Furstenberg-Katznelson-Ornstein result by proving that for any transformation preserving measure, the $L^2$-norm convergence for $T_i=T^i$ holds . In 1984, J. Conze and E. Lesigne in  \cite{Conze-Le} gives a positive answer for the case $k=2$. Under some extra ergodicity assumptions, Conze-Lesigne result was
extended to the case $k=3$ by Zhang in \cite{Zhang} and for any $k \geq 2$ by Frantzikinakis and Kra in \cite{Fran-Kra}. Without these assumptions, this result was proved by T. Tao in \cite{Tao}. Subsequently, T. Austin in \cite{Austin} gives a joining alternative proof of Tao result, and recently, M. Walsh extended Tao result by proving that the $L^2$-norm convergence holds for the maps $(T_i)_{i=1}^{k}$ generate a nilpotent
group \cite{Walsh}. This solve a Bergelson-Leibman conjecture stated in \cite{Bergelson-Lei}. Therein, the authors produced a counter-examples of maps generated a solvable group for which the $L^2$-norm convergence does not hold.

For the pointwise convergence, partial results were obtained in \cite{Lesigne}, \cite{AssaniW}, \cite{Assani-S} and \cite{XiangDong} under some ergodic and spectral assumptions. In a very recent preprint \cite{Assani}, I. Assani observe that the action of the maps $T_i$ induced a dynamical system $(X^k,\bigotimes_{i=1}^k\cb,\nu,\phi)$ where $\phi(x_1,\cdots,x_k)=(T_ix_i)_{i=1}^{k}$ is a non-singular map with respect to the probability measure $\nu$ given by
$$\nu(A_1 \times A_2 \times \cdots A_k)=\frac13\sum_{n \in \Z}\frac1{2^{|n|}}\mu_{\Delta}(\phi^{-n}(A_1 \times A_2 \times \cdots A_k)),$$
with $\mu_{\Delta}$ is a diagonal probability measure on $X^k$ define on the rectangle $A_1 \times A_2 \times \cdots \times A_k$ by
$$\mu_{\Delta}(A_1\times A_2 \times \cdots A_k)=\mu(A_1 \cap A_2 \cap \cdots A_k).$$

It is easy to see that the Radon-Nikodym derivative of the pushforward measure of $\nu$ under $\phi$ satisfy
$$\frac12 \leq \frac{d\nu \circ \phi}{d\nu} \leq 2.$$

Unfortunately as we shall see in section 5, the strategy of \cite{Assani} fails since the maximal  ergodic inequality does not hold for the non-singular dynamical system $(X^k,\cb^k,\nu,\phi)$. Nevertheless, applying Tuleca's theorem \cite{Tulcea}, we are able to conclude that the individual ergodic theorem holds for $\phi$ in the sense of Tuleca.\\

We associate to $\phi$ the Koopman operator $U_{\phi}$ defined by $U_{\phi}(\overline{f})=\overline{f}\circ \phi,$ where $\overline{f}$ is a measurable function on $X^k$. Since $U_{\phi}$ maps $L^{\infty}$ on $L^{\infty}$ and $\phi$ is non-singular,
the adjoint operator $U_{\phi}^*$ acting on $L_1$ can be defined by the relation
$$\int U_{\phi}(\overline{f}).\overline{g} d\nu = \int \overline{f}. U_{\phi}^*(\overline{g}) d\nu,$$
for any $\overline{f} \in L^{\infty}$ and $\overline{g} \in L^1$. For simplicity of notation, we write $\phi^*$ instead of
$U_{\phi}^*$ and $\phi$ instead of $U_{\phi}$ when no confusion can arise.

$\phi^*$ is often referred to as the Perron-Frobenius operator or transfer operator associated with $\phi$.
The basic properties of $\phi^*$ can be found in \cite{Krengel} and \cite{aaronson}. It is well known that the pointwise ergodic theorem for $\phi$ can be characterized by $\phi^*$.
Indeed, Y. Ito established that the validity of the $L^1$-mean theorem for $\phi^*$ implies the validity of the pointwise ergodic theorem for $\phi$ from $L^1$ to $L^1$ \cite{Ito}. Moreover, the subject has been intensively studied by many authors (
Ryll-Nardzewski \cite{Ryll}, Tuleca \cite{Tulcea}, Hopf, Choksi \cite{Hopf}, Assani \cite{Assani85}, Assani-Wo\'{s} \cite{Assani-W}, Ortega Salvador \cite{Ortega}, R. Sato
 \cite{Sato1}, \cite{Sato2}). Here, we will use and adapt the Ryll-Nardzewski approach \cite{Ryll} and its extension \cite{Tulcea}. A nice account on the previous results can be found in \cite[p.31-61]{Friedman}.\\

Notice that the $L^2$-norm convergence implies that for any $k$-uplet of Borel set $(A_i)_{i=1}^{k}$, we have
$$\lim_{N \longrightarrow +\infty}\frac1{N}\sum_{n=1}^{N}\mu_{\Delta}\Big(\phi^{-n}\Big(\prod_{i=1}^{k}A_i\Big)\Big)=
\mu^{F}\Big(\prod_{i=1}^{k}A_i\Big),$$
where $\mu^{F}$ is an invariant probability measure under the actions of $(T_i)$. Following T. Austin, $\mu^{F}$ is called a Furstenberg self-joining of $(T_i)$ \cite{Austin},\cite{Austin-P}, \cite{Austin-C}. Therein, T. Austin stated the multiple recurrence theorem in the following form
$$\forall (A_i)_{i=1}^k \in \cb^k,~~~~
\mu^{F}(A_1\times A_2 \times \cdots  \times A_k)=0\Longrightarrow \mu_{\Delta}(A_1\times A_2 \times \cdots  \times A_k)=0,$$

 We further point out that if the maps $T_i=T^i$ and $T$ is a weakly mixing map, then $\mu^{F}=\bigotimes_{i=1}^{k}\mu$, that is, $\mu^{F}$ is singular with respect to $\mu_{\Delta}$ and this is not incompatible with the multiple recurrence theorem, since the absolutely continuity holds only one the sub-algebra generated by the rectangle Borel sets.\\

The proof of T. Austin is based on the description of $\mu^F$ in terms of the measure $\mu$ and various partially-invariant factors of $\cb$. This is done using a suitable extension of the originally-given system $(X,\cb,\mu, T_i)$. Here, we will further need the method used by Hansel-Raoult in \cite{Hansel-Ra} and its generalization to $\Z^d$ action obtain by B. Weiss
in \cite{Weiss}. Therein, the authors gives a generalization of the Jewett theorem using the Stone representation theorem combined with some combinatorial arguments. The proof given by Weiss yields a "uniform" extension of Rohklin towers lemma. This allows us, under a suitable assumption, to produce a strictly ergodic topological model for which we are able to show under some condition that one can drop a subsequence of Furstenberg averages for which the convergence almost everywhere holds.\\

 Summarizing, our proof is essentially based on two kind of arguments. On one hand, the Assani observation \cite{Assani} combined with some ideas from Ryll-Nardzewski result \cite{Ryll} and the Tuleca result \cite{Tulcea} and on the other hand on the $\cc$-method introduced by Austin \cite{Austin} combined with Hansel-Raoult-Weiss procedure \cite{Hansel-Ra}.\\

Let us mention that at know, in the general setting, the problem of the pointwise convergence of the Furstenberg ergodic averages still open. The only know results are Bourgain double ergodic theorem\cite{BourgainD} and the distal case \cite{XiangDong}. The paper is organized as follows\\

In section 2, we state our main result and we recall the main ingredients need it for the proof. In section 3, we establish under our assumptions the $L^2$-norm convergence. In section 4, we recall the Stone representation theorem and Hansel-Raoult-Weiss procedure used to produce a strictly ergodic topological model.  In section 5, we establish that the pointwise convergence of the multiple ergodic averages holds along some subsequence, and in the end of the section, we prove that there exists a non-singular dynamical system for which the maximal ergodic inequality does not hold. Finally, in section 6, we give a proof of our main results.

\section{Main result}\label{S2}
Let $(X,\cb,\mu)$ be a Lebesgue probability space, that is, $X$ is a Polish space (i.e. metrizable separable and complete), whose Borel $\sigma$-algebra $\cb$ is complete with respect to the probability measure $\mu$ on $X$. The notion of Lebesgue space is due to Rokhlin \cite{Rokhlin}, and it is well known \cite{Petersen} that a Lebesgue probability space is isomorphic (mod 0) to ordinary Lebesgue space $([0,1),\cc,\lambda)$ possibly together with countably many atoms, that is, there are $x_0,x_1, \cdots,\in X$, $X_0 \subset X$,
$Y_0 \subset [0,1)$, and $\phi : X_0 \bigcup \{x_i\} \longrightarrow Y_0$, which is invertible such that the pushforward measure of $\mu$ under $\phi$ is $\lambda$ with $\mu(X_0 \bigcup \{x_i\})=1=\lambda(Y_0)=1$ . Here, we will deal only with non-atomic Lebesgue space.\\

A dynamical system is given by $(X,\cb,\mu,T)$ where $(X,\cb,\mu)$ is a Lebesgue space  and
 $T$ is an invertible bi-measurable transformation which preserves the probability measure $\mu$.\\

In this context, we state our first main result in the soft form as follows
\begin{Th}\label{main2}There is a non-singular dynamical system for which the maximal ergodic inequality doesn't holds.
\end{Th}

Our second main result can be stated as follows
\begin{Th}\label{main1}Let $k \in \N^*$ and $(X,\cb,\mu,T_i)_{i=1}^{k}$ be a finite family of dynamical systems where $(X,\cb,\mu)$ is Lebesgue probability space, and assume that $T_1,T_2,\cdots,T_k$ are commuting weakly mixing transformations on $X$ such that for any $i \neq j$, the map $T_i \circ T_j^{-1}$ is ergodic. Then, there is a subsequence $(N_l)$ such that for
every $f_i \in L^{\infty}(\mu)$, $i=1,\cdots,k$, the averages
$$ \frac1{N_l}\sum_{n=1}^{N_l}\prod_{i=1}^{k}f_i(T_i^nx)$$
 converges almost everywhere to $\ds \prod_{j=1}^{k}\int f_j d\mu$.
\end{Th} 

The proof is based on the Tuleca's result \cite{Tulcea} and some ideas from the Ryll-Nardzewski approach \cite{Ryll} combined with the Cantor diagonal method, and the machinery of $\cc$-systems introduced by T. Austin. We remind that this machinery allows T. Austin to obtain a joining proof of the Tao theorem on the $L^2$-norm convergence of the Furstenberg ergodic averages. We will use the $\cc$-systems machinery in the section 3. Here, we recall the Ryll-Nardzewski theorem \cite{Ryll}.

\begin{Th}[Ryll-Nardzewski \cite{Ryll}]\label{satolem} Let $(X,\cb,\nu)$ be a $\sigma$-finite measure and $\phi$ an invertible map on $X$ such that the pushforward measure of $\nu$ under $\phi$ is absolutely continuous with respect to $\nu$.
 Then the following conditions are equivalent.
\begin{enumerate}[(a)]
  \item The operator $Tf=f \circ \phi$ satisfies the pointwise ergodic theorem
  from $L^{1}(\nu)$ to $L^{1}(\nu)$, that is, for any function $f \in L^1(\nu)$ there is a function $g \in L^1(\nu)$ such that
  $$\lim_{N \longrightarrow +\infty}\frac1{N}\sum_{n=1}^{N}(T^{n}f)(x)=g(x)~~~~~~~~ \nu{\textrm {.a.e.}}~$$
  \item There is a constant $K$ such that for any Borel set A and each $Y$ with \linebreak $\nu(Y)<+\infty$ we have
  $$ \limsup_{N \longrightarrow +\infty}\frac1{N}\sum_{n=1}^{N}\nu(Y \cap \phi^{-n}A)\leq K \nu(A),$$
\end{enumerate}
\end{Th}
The condition (b) is called Hartman condition, and the proof of Theorem \ref{satolem} is essentially based on the notion of Mazur-Banach limit. Precisely, using the Mazur-Banach limit, Ryll-Nardzewski proved that there is a $\sigma$-finite measure $\rho$ such that, for any Borel set $A$, we have
\begin{enumerate}[(i)]
  \item $0 \leq \rho(A) \leq K \nu(A);$
  \item if $A=\phi^{-1}A$ then $\rho(A)=\nu(A)$;
  \item $\rho(\phi^{-1}A)=\rho(A)$.
\end{enumerate}
 This insure that $\rho$ is $\phi$-invariant and we can apply the Birkhoff ergodic theorem to conclude. For more details and the rest of the proof, we refer the reader to the Ryll-Nardzewski paper \cite{Ryll}. Let us further point out that therein,
Ryll-Nardzewski produce a counter-example for which the pointwise ergodic theorem in $L^1(\nu)$ doesn't imply the ergodic theorem in $L^1(\nu)$. We remind that the ergodic theorem in $L^1(\nu)$ holds, if for any $f \in L^1(\nu)$, there is a function $g \in L^1(\nu)$ such that
$$\Big\|\frac1{N}\sum_{n=1}^{N}f(\phi^n(x))-g(x)\Big\|_1 \tend{N}{+\infty}0.$$

Following Ryll-Nardzewski ideas, we choose a non-decreasing sequence
 of Borel set $(A_m)$ such that $\lim A_m=X$, $\nu(A_m)<+\infty$, and for any Borel set $A$, we define the sequence $\nu_m(A)$ by
 $$\nu_m(A)=\frac1{N}\sum_{n=0}^{N-1}\nu(\phi^n(A)\cap A_m).$$
Hence $(\nu_m(A))_{m \geq 0}$ is a bounded sequence, and it is an easy exercise to see that we can extend the operator ${\boldsymbol{\lim}}$ on the space of real bounded sequences to obtain a bounded operator on $\ell^{\infty}$ by Hahn-Banach theorem. We denote such operator by $\rm{\bf {MBlim}}$. This allows us to define a sequence of finite measures $\nu_m$ on $X$ given by
 $$\nu_m(A)={\rm{\bf {MBlim}}}\Big(\frac1{N}\sum_{n=0}^{N-1}\nu(\phi^n(A)\cap A_m)\Big).$$
 It follows that if $A=\phi(A)$ then, for any $m \in \N$, we have
$\nu_m(A)=\nu(A \cap A_m).$ We further have, for any Borel set $A$,
$$\nu_m(A) \leq K \nu(A),$$
and
\begin{eqnarray*}
\nu_{m}(\phi(A))&=&{\rm{\bf {MBlim}}}\Big(\frac1{N}\sum_{n=0}^{N-1}\nu(\phi^{n+1}(A)\cap A_m)\Big)\\
&=&{\rm{\bf {MBlim}}}\Big(\frac1{N}\sum_{n=0}^{N-1}\nu(\phi^n(A)\cap A_m)-\frac{\nu(A \cap A_m)}{N}+\frac{\nu(\phi^{N}(A) \cap A_m)}{N}\Big)\\
&=&{\rm{\bf {MBlim}}}\Big(\frac1{N}\sum_{n=0}^{N-1}\nu(\phi^n(A)\cap A_m)\Big)\\
&=&\nu_m(A),
\end{eqnarray*}
since
$${\rm{\bf {MBlim}}}\Big(-\frac{\nu(A \cap A_m)}{N}+\frac{\nu(\phi^{N}(A) \cap A_m)}{N}\Big)=0.$$
That is, $\nu_m$ is invariant under $\phi$. Now, the sequence $(\nu_m(A))$ is a bounded non-decreasing sequence. Therefore,
we can put
$$\rho(A)=\lim_{m \longrightarrow +\infty}\nu_m(A),$$
and it is easy to check that (i), (ii) and (iii) holds.
\begin{rem}The condition (i) insure that $L^1(\nu) \subset L^1(\rho)$.
\end{rem}
We remind that the Mazur-Banach limit operator ${\rm{\bf {MBlim}}}$ satisfies the following properties:
\begin{enumerate}
\item ${\rm{\bf {MBlim}}}$ is positive: ${\rm{\bf {MBlim}}}(x)\geq 0$ for every $x \in \ell^{\infty}(\N)$ with $x_n \geq 0.$
\item  ${\rm{\bf {MBlim}}}$ is normalized: ${\rm{\bf {MBlim}}}(\boldsymbol{1})=1$ where  $\boldsymbol{1}=(1,1,\cdots).$
\item ${\rm{\bf {MBlim}}}$ is shift invariant: ${\rm{\bf {MBlim}}}(Sx)={\rm{\bf {MBlim}}}(x)$ where
$Sx=(x_{n+1})_{n \geq 0}$.
\item ${\rm{\bf {MBlim}}}$ has norm one: $\big|{\rm{\bf {MBlim}}}(x)\big|\leq\|x\|_{\infty}$ for every $x \in \ell^{\infty}$.
\item For any $x \in \ell^{\infty}(\N)$:
$$\liminf(x_n) \leq {\rm{\bf {MBlim}}}(x) \leq \limsup(x_n).$$
\end{enumerate}
For more details on the connection between the ergodic theory and the Mazur-Banach limit theory, we refer the reader to \cite{Sucheston1}, \cite{Sucheston2}, \cite{Goffman} and \cite{Jerison}.\\

We notice that Ryll-Nardzewski approach yields the existence of the absolutely finite invariant measure. In the case of the existence of equivalent $\sigma$-finite invariant measure, we have the following result due to Tulcea \cite{Tulcea} (see also \cite[p.55]{Friedman})
\begin{Th}\label{Tuleca}Let $T$ be a non-singular map on a Lebesgue space $(X,\cb,\nu)$, and assume that $T$ admits a $\sigma$-finite invariant measure $\overline{\nu}$ equivalent to $\nu$. Then, for each $f \in L^1(\nu)$ there exists $f^* \in L^1(\nu)$ such that
\begin{enumerate}[(a)]
  \item $\ds\frac1{N}\sum_{n=1}^{N}f\circ T^n.\frac{d\nu \circ T^n}{d\nu}(x) \tend{n}{+\infty}f^*$ a.e.,
   \item $\ds f^* \circ T. \frac{d\nu \circ T}{d\nu}(x)=f^*$ a.e..
\end{enumerate}
\end{Th}
Tulcea theorem can be derived from the Hurewicz-Halmos-Oxtoby ratio ergodic theorem \cite{Halmos}, and it is well-known that the ratio ergodic theorem is intimately related to the following Hopf decomposition theorem

\begin{Th}[Hopf decomposition theorem \cite{Hopf}]Let $(X,\cb,T,\sigma)$ be a non-singular dynamical system. Then, there exists
a decomposition $X=X_c \cup X_d$, such that $X_c,X_d$ are measurable, invariant under $T$, and such that $T|_{X_c}$ is incompressible, $T|_{X_d}$ completely dissipative. We further have
$$X_c=\Big\{x~~:~~\sum_{n \geq 0}\frac{d\sigma \circ T^{n}}{d\sigma}=+\infty\Big\},$$
$$X_d=\Big\{x~~:~~\sum_{n \geq 0}\frac{d\sigma \circ T^{n}}{d\sigma}<+\infty\Big\}.$$
\end{Th}
We remind that the dynamical system $(Y,\cb,T,\sigma)$ is incompressible if for any  Borel set $A$ such that
$T^{-1}A \subset A$, we have $\sigma(A \Delta T^{-1}A)=0$. The dynamical system $(Y,\cb,T,\sigma)$ is completely dissipative if there is a Borel set $W$ such that for any $n \neq m$, $T^nW \cap T^mW=\emptyset$ and $Y=\cup_{n \in \Z} T^nW,$ and if we put
$$R_N(T,f)=\frac{\ds \sum_{n=0}^{N-1}f\circ T^n(x)\frac{d\sigma \circ T^{n}}{d\sigma}(x)}
{\ds \sum_{n=0}^{N-1}\frac{d\sigma \circ T^{n}}{d\sigma}(x)},$$
for any $f \in L^1(\sigma)$. Then
\begin{Th}[Hurewicz-Halmos-Oxtoby ergodic theorem \cite{Halmos})] On $X_c$, for all $f \in L^1(\sigma)$, $R_N(T,f)$
converge a.e. to a limit function $f^* \in L^1(\sigma)$; $f^* \circ T=f^*$ a.e. and
$$\int_{X_c}f^* d\sigma=\int_{X_c}f d\sigma.$$
On $X_d$, for all $f \in L^1(\sigma)$, $R_N(T,f)$
converge a.e. to the ratio of the two convergent series $\ds \sum_{n=0}^{+\infty}f\circ T^n(x)\frac{d\sigma \circ T^{n}}{d\sigma}(x)$
and $\ds \sum_{n=0}^{+\infty}\frac{d\sigma \circ T^{n}}{d\sigma}(x).$
\end{Th}

Finally, our analogous result is related to the irregular set of the continuous functions. In the multifractal analysis theory, the irregular set of the measurable function $f$ is given by
$$X_{f}=\Big\{x \in X~~:~~ \liminf \frac1{N}\sum_{n=1}^{N}f \circ T^n(x) < \limsup \frac1{N}\sum_{n=1}^{N}f \circ T^n(x) \Big\}.$$ As a consequence of Birkhoff's ergodic theorem, we have
\begin{Th}\label{fractal}
For a continuous map $T$ on a compact metric space $X$, if the
function $f$ is continuous, then $\rho(X_f) = 0$ for any $T$-invariant finite measure
$\rho$ on $X$.
\end{Th}
We further have that the irregular set satisfy
$$X_{f}=\Big\{x \in X~~:~~  \Big(\frac1{N}\sum_{n=1}^{N}f \circ T^n(x)\Big) {\textrm{~~is~not~a~Cauchy~sequence}} \Big\},$$
that is,
$$X_{f}=\bigcup_{k \geq 1 }\bigcap_{n \in \N}\bigcup_{M,N \geq n}\Big\{x \in X~~:~~ \Big|F_N(x)-F_M(x)\Big| >\frac1{k} \Big\},$$
where
$$F_N(x)=\frac1{N}\sum_{n=1}^{N}f \circ T^n(x).$$
Put
$$O_{n,k}=\bigcup_{M,N \geq n}\Big\{x \in X~~:~~ \Big|F_N(x)-F_M(x)\Big| >\frac1{k} \Big\}.$$
Obviously, $O_{n,k}$ is open set, for any $n,k \in \N^*$. Hence $X_f$ is a Borel set. Applying the same reasoning combined with the separability of the space of continuous functions $\mathcal{C}(X)$, one can see that the set of generic points is a Borel set and it has a full measure for any finite $T$-invariant measure on $X$ (a point $x$ is generic if, for any continuous function
$f$, the Birkhoff sum $\frac1N\sum_{n=1}^N f(T^nx)$ converge). For the more details, we refer the reader to \cite[Chap. 11]{Aoki}.

\section{$L^2$-norm convergence of Furstenberg averages}
We start by proving the following
\begin{Th}[Frantzikinakis-Kra \cite{Fran-Kra}]\label{fact1}Let $k \in \N^*$ and $(X,\cb,\mu,T_i)_{i=1}^{k}$ be a finite family of dynamical systems where $\mu$ is a probability measure space, and  $T_1,T_2,\cdots,T_k$ are commuting weakly mixing transformations on $X$ such that for any $i \neq j$, the map $T_i \circ T_j^{-1}$ is ergodic. Then, for
every $f_i \in L^{\infty}(\mu)$, $i=1,\cdots,k$, the averages
$$ \frac1{N}\sum_{n=1}^{N} \prod_{i=1}^{k}f_i(T_i^nx)$$
 converge in $L^2(X,\mu)$ to $\ds \prod_{j=1}^{k}\int f_j d\mu$.
\end{Th}
The proof is based on van der Corput trick and the $\cc$-sated systems method introduced by T. Austin. In our case the $\cc$-sated systems are trivial, and the \linebreak van der Corput lemma can be stated in the following form
\begin{lem}[van der Corput \cite{Berg}]\label{van}
Let $(u_n)$ be a bounded sequence in a Hilbert space. Then,
$$\limsup\Big\|\frac1{N}\sum_{n=0}^{N-1}u_n\Big\|^2 \leq
\limsup \frac1{H}\sum_{h=0}^{H-1}\limsup\Big|\sum_{n=0}^{N} \langle u_{n+h}, u_n \rangle \Big|.$$
\end{lem}
\noindent{}As a simple consequence it follows that if
$$\lim_{H \longrightarrow +\infty}\lim_{N \longrightarrow +\infty}
\frac{1}{H}\sum_{h=0}^{H-1}\sum_{n=0}^{N}\langle u_{n+h}, u_n \rangle=0,$$
then
$$\Big\|\frac1{N}\sum_{n=0}^{N-1}u_n\Big\|\tend{N}{+\infty}0.$$
We remind that if $(T_i)_{i=1}^{k}$ are a commuting maps on $X$ then the associated $\cc$-systems are the dynamical systems
$(X,\cc,\mu,T_i)$, $i=1,\cdots,k,$ for which
$$\cc=\ci_{T_1} \vee \ci_{T_2T_1^{-1}}\vee \ci_{T_2T_1^{-1}} \cdots \vee \ci_{T_kT_1^{-1}} ,$$
where, for any transformation $S$, $\ci_{S}$ is the factor $\sigma$-algebra of $S$-invariant Borel sets, that is,
$$\ci_S=\Big\{A~~:~~\mu(A \Delta S^{-1}A)=0\Big\}.$$
Notice that under our assumption the $\sigma$-algebra $\cc$ is trivial. The key notion in the Austin proof is the notion of
$\cc$-sated system defined as follows \cite{Austin-P}
\begin{Def}Let $k$ be a integer such that $k \geq 2$. The system $(X,\cb,(T_i)_{i=1}^{k}),$ is $\cc$-sated if any joining $\lambda$ of $X$ with any $\cc$-system  $Y$ is relatively
independent over the largest $\cc$-factor $X_{\cc}$ of $X$, that is, for any bounded measurable function $f$ on $X$, we have
$$\E_{\lambda}(f(x)|Y)=\E_{\lambda}(\E_{X}(f(x)|X_{\cc})|Y),$$
Where $\E_{\lambda}(.|\bullet)$ is a conditional expectation operator.
\end{Def}
\noindent{}We denote the expectation operator by
$$\E(f)=\int f d\mu,~~~f \in L^2(X).$$
Following this setting, the Tao $L^2$-norm convergence theorem can be stated as follows
\begin{Th}[Austin-Tao \cite{Austin},\cite{Tao}]\label{Tao-Austin} Let $k \geq 1$, and $(X,\cb,T_i)$, $i=1,\cdots,k,$ be a $\cc$-sated system. Then, for any $f_1,f_2,\cdots,f_k$ functions in $L^{\infty}(X)$,
$$\E_{X}(f_1|X_{\cc})=0 \Longrightarrow \Big\|\frac1{N}\sum_{n=1}^{N}f_1 \circ T_1 \cdots f_k \circ T_k\Big\|_2
\tend{N}{+\infty}0.$$
\end{Th}
Now, we are able to give the proof of Theorem \ref{fact1}.
\begin{proof}[\bf{Proof of Theorem \ref{fact1}}] We use induction on $k$ to prove the theorem. The statement is obvious for $k =1$. For the case $k=2$, by our assumption the $\cc$-system is trivial and we can write
\begin{eqnarray*}
&&\Big\|\frac1{N}\sum_{n=1}^{N} f_1(T_1^nx)f_2(T_2^nx)-\E(f_1)\E(f_2)\Big\|_2 \\
&\leq& \Big\|\frac1{N}\sum_{n=1}^{N}(f_1-\E(f_1))(T_1^nx)f_2(T_2^nx)+
\frac1{N}\sum_{n=1}^{N}f_2(T_2^nx) \E(f_1)
-\E(f_1)\E(f_2)\Big\|_2\\
&\leq&\Big\|\frac1{N}\sum_{n=1}^{N}(f_1-\E(f_1))(T_1^nx)f_2(T_2^nx)\Big\|_2+
\Big\|\frac1{N}\sum_{n=1}^{N}f_2(T_2^nx) \E(f_1)
-\E(f_1)\E(f_2)\Big\|_2
\end{eqnarray*}
Hence, by the $L^2$-norm convergence and the von Neumann ergodic theorem, we have
$$\Big\|\frac1{N}\sum_{n=1}^{N} f_1(T_1^nx)f_2(T_2^nx)-\E(f_1)\E(f_2)\Big\|_2
\tend{N}{+\infty}0.$$
In the general case, applying the van der Corput Lemma \ref{van}, the $L^2$ convergence of Furstenberg average of order $k$, can be reduced in the class of $\cc$-systems to the case of $k-1$ commuting maps. In this case, the $L^2$ convergence gives
\begin{eqnarray*}
\E(f_1)=0 \Longrightarrow
\Big\|\frac1{N}\sum_{n=1}^{N} f_1(T_1^nx)f_2(T_2^nx)\cdots f_k(T_k^nx)\Big\|_2
\tend{N}{+\infty}0.
\end{eqnarray*}
Therefore, suppose that the result holds for some integer $k \geq 1$,
and assume that $(X,\mu,T_1,\cdots,T_l)$ is a system of order $k+1$. Then,
\begin{eqnarray*}
&&\Big\|\frac1{N}\sum_{n=1}^{N} f_1(T_1^nx)f_2(T_2^nx)\cdots f_{k+1}(T_{k+1}^nx)-\E(f_1)\E(f_2)\cdots\E(f_{k+1})\Big\|_2 \\
&=& \Big\|\frac1{N}\sum_{n=1}^{N}(f_1-\E(f_1))(T_1^nx)f_2(T_2^nx)\cdots f_{k+1}(T_{k+1}^nx)+
\\&&\E(f_1)\frac1{N}\sum_{n=1}^{N}f_2(T_2^nx)\cdots f_{k+1}(T_{k+1}^nx)
-\E(f_1)\E(f_2)\cdots\E(f_{k+1})\Big\|_2\\
&\leq&
\Big\|\frac1{N}\sum_{n=1}^{N}(f_1-\E(f_1))(T_1^nx)f_2(T_2^nx)\cdots f_{k+1}(T_{k+1}^nx)\Big\|_2
\\&+&|\E(f_1)|\Big\|\frac1{N}\sum_{n=1}^{N}f_2(T_2^nx)\cdots f_{k+1}(T_{k+1}^nx)
-\E(f_2)\cdots\E(f_{k+1})\Big\|_2 \tend{N}{+\infty}0.
\end{eqnarray*}
This complete the proof of the theorem.
\end{proof}
As a consequence we deduce the following result
\begin{Cor}\label{Cor1}Let $k \in \N^*$ and $(X,\cb,\mu,T_i)_{i=1}^{k}$ be a finite family of dynamical systems where $\mu$ is a probability measure space, assume that $T_i$, $i=1 \cdots k$ are commuting weakly mixing transformations on $X$. Then, for
every $A_i \in \cb $, $i=1,\cdots,k$, the averages
$$ \frac1{N}\sum_{n=1}^{N}\mu_{\Delta}(T_1^{-n}(A_1)\times T_2^{-n}(A_2) \times \cdots \times T_k^{-n}( A_k))$$
 converge to $\mu(A_1)\mu(A_2) \cdots  \mu(A_k)$.
\end{Cor}
\begin{rem}A soft proof of the proposition \ref{Tao-Austin} can be obtained as a consequence of proposition 2.3 in \cite{Fran-Kra}.
\end{rem}
Form this we deduce the following lemma
\begin{lem}\label{curcial}
Let $k \in \N^*$ and $(X,\cb,\mu,T_i)_{i=1}^{k}$ be a finite family of dynamical systems where $\mu$ is a probability measure space, and assume that $T_1,T_2,\cdots,T_k$ are commuting weakly mixing transformations on $X$. Then, for
every $f_i \in L^{\infty}(\mu)$, $i=1,\cdots,k$, the averages
$$ \frac1{N}\sum_{n=1}^{N}\int\prod_{i=1}^{k}f_i(T_i^nx)d\nu$$
 converge to $\ds \prod_{i=1}^{k}\mu(f_i)$.
\end{lem}
\begin{proof}It is suffice to prove the lemma for any finite family of Borel set $(A_i)_{i=1}^{k}$. Indeed, by Corollary \ref{Cor1}, the sequence
$\ds \frac1{N}\sum_{n=1}^{N}\mu_{\Delta}\Big(\phi^{-n}(\overline{A})\Big)$
converge to $\ds \bigotimes_{i=1}^{k}\mu(\overline{A})$, where $\overline{A}=A_1\times A_2 \cdots \times A_k$, and $\phi=T_1\times T_2 \times \cdots \times T_k$. This gives us that the convergence holds when replacing $\mu_{\Delta}$ by
$\mu_{\phi^{j}\Delta}$, for any $j \in \Z$, where $\mu_{\phi^{j}\Delta}$ is the pushforward measure of $\mu_{\Delta}$ under $\phi^j$. Therefore, for any $M \in \N$, the convergence holds for $\nu_M$ where
$$\nu_M(\overline{A})=\frac1{3}\sum_{|n| \leq M}\frac1{2^{|n|}}\mu_{\phi^{n}\Delta}(\overline{A}).$$
Now, let $\varepsilon>0$ then there is a positive integer $M_0$ such that for any $M \geq M_0$, we have
$$|\nu_M(\overline{A})-\nu(\overline{A})| \leq \varepsilon,~~~~\forall \overline{A}\in \ca^k.$$
Hence, for any $N \in \N$,
\[
\Big|\frac1{N}\sum_{n=1}^{N}\nu_M(\phi^{-n}(\overline{A}))-
\frac1{N}\sum_{n=1}^{N}\nu\big(\phi^{-n}(\overline{A})\big)\Big|<\varepsilon.
\]
By letting $N$ and $M$ goes to infinity we obtain
\[
\bigotimes_{i=1}^{k}\mu(\overline{A})-\varepsilon \leq
\liminf_{N \longrightarrow +\infty}\frac1{N}\sum_{n=1}^{N}\nu(\phi^{-n}(\overline{A}))
\leq \limsup_{N\longrightarrow +\infty}\frac1{N}\sum_{n=1}^{N}\nu(\phi^{-n}(\overline{A}))
\leq \bigotimes_{i=1}^{k}\mu(\overline{A})+\varepsilon.
\]
Since $\varepsilon$ was chosen arbitrarily, we conclude that
 $$\frac1{N}\sum_{n=1}^{N}\nu\big(\phi^{-n}(\overline{A})\big) \tend{N}{+\infty} \bigotimes_{i=1}^{k}\mu\big(\overline{A}\big).$$
 This proves the lemma.
\end{proof}

\section{Stone representation theorem and Furstenberg averages}
Let us consider a dynamical system $(X,\cb,\mu,T)$. Then,
there exists a countable algebra $\ca$ dense in $\cb$ for the pseudo-metric
$d(A,B)=\mu(A \Delta B)$, $A$ and $B$ in $\ca$. We further have that $\ca$ separates the points of X, that is, for each
$x,y \in X$  with $x \neq y$, there is $A \in \ca$ such
that either $x \in A, y \not \in A $ or else $y \in A, x \not \in A$. Hence, by the Stone representation theorem, we  associate to ${\ca}$ a Stone algebra $\widehat{\ca}$ on the set $\widehat{X}$ of all ultrafilter on $X$ such that for any  $A \in \ca$,
$\widehat{A}=\{\cu_A \in \widehat{X} / A \in \cu_A\}$. Consequently $\widehat{\ca}=\{\widehat{A}, A\in \ca\}$ is algebra of subsets of $\widehat{X}$, which is isomorphic to $\ca$. $\widehat{\ca}$ is called the Stone algebra.\\

Assuming that $T$ is ergodic, Hansel and Raoult proved in \cite{Hansel-Ra} that there is a dense and invariant countable algebra $\ca_T$ such that for any $A$ of $\ca_T$ we have
$$\Big\|\frac1{N}\sum_{n=1}^{N}\1_{A}(T^nx)-\mu(A)\Big\|_{\infty} \tend{N}{+\infty}0,$$
that is, $A$ is a {\it {uniform ergodic set}}.\\

Their result was extended to the ergodic $\Z^d$-action by B. Weiss in \cite{Weiss}. Indeed, B. Weiss proved that if
$\overline{T}=(T_i)_{i=1}^{d}$ is a generator of ergodic $\Z^d$-action then there is a dense and $\overline{T}$-invariant countable algebra $\ca_{\overline{T}}$ such that all elements $A$ in $\ca_{\overline{T}}$ are uniform ergodic sets, that is,
$$\Big\|\frac1{|R_n|}\sum_{\overline{i} \in R_n}\1_A(\overline{T}^{\overline{i}}x)-\mu(A)\Big\|_{\infty}\tend{n}{+\infty}0,$$

where $R_n$ is the square $\{\overline{i} \in \Z^d: |\overline{i}|_{\infty} \leq n\}$.\\

 Applying a Stone representation theorem to Hansel-Raoult-Weiss algebra, and letting $\widehat{X}$ be equipped with the topology which has $\widehat{\ca}$ as a base of clopen sets. It follows that $\widehat{X}$ is metrizable space (since $\widehat{\ca}$ is countable), compact (by the standard ultrafilter lemma) and totaly disconnected. $\widehat{X}$ is called the Stone space of
$\ca$. Furthermore, the probability $\mu$ induces a mapping $\mu'$ from $\widehat{\ca}$ into [0,1], such that
\begin{enumerate}[(1)]
\item for any $A \in \widehat{\ca}\setminus\{\emptyset\}$, $\mu'(A)>0$,
\item $\mu'$ is $\sigma$-additive on $\widehat{\ca}$, and $\mu'(\widehat{X})=1$.
\end{enumerate}
Hence, by  Carath\'eodory's extension theorem, $\mu'$ has an unique extension as a probability measure $\widehat{\mu}$ on the Borel $\sigma$-algebra $\widehat{\cb}$ generate by $\widehat{\ca}$. We further have that for every non-empty open subset
$O \subset \widehat{X}$, $\widehat{\mu}(O)>0$. This gives in particular that for any $\widehat{A}\in \widehat{B}$, if
$\widehat{\mu}(\widehat{A})=1$ then $\widehat{A}$ is a dense in $\widehat{X}$.\\

The $\Z^d$-action with generators $\overline{T}$ induces a $\Z^d$-action with generators $\widehat{T}=(\widehat{T}_i)_{i=1}^{d}$ such that, by construction and due to intrinsic properties of Lebesgue space, the two action are isomorphic and for each $i$, $\widehat{T}_i$ is a homeomorphism on $\widehat{X}$. We further have that for any
$\widehat{A}$, the sequence of continuous function
$$\frac1{|R_n|}\sum_{\overline{n}\in R_N}\1_{\widehat{A}}\circ \widehat{T}^{\overline{n}}$$
converges uniformly to $\widehat{\mu}(\widehat{A})$. Then, it follows by Birkhoff's ergodic theorem that the $\Z^d$-action with generators $\widehat{T}=(\widehat{T}_i)_{i=1}^{d}$ is {\it{strictly ergodic}}, that is, $\widehat{\mu}$ is a unique probability measure $\widehat{T}$-invariant with $\widehat{\mu}(O)>0$ for every nonempty open set $O\subset \widehat{X}$.  The  topological model $(\widehat{X},\widehat{\cb},\widehat{T})$ is called the Stone-Jewett-Weiss topological model. For the case $d=1$, we call it the Stone-Jewett-Hansel-Raoult topological model.\\

Finally, let us denote by $B(\ca)$ the Banach space of all scalar-valued functions that are uniform limits of sequences of $\ca$-measurable step functions, equipped with the supremum norm. Then, $B(\ca)$ is isometrically isomorphic to
$\mathcal{C}( \widehat{X})$, the Banach space of all continuous functions on the Stone space $\widehat{X}$.

\section{Proof of the main results}
For the proof of our main results, we shall consider the Stone-Jewett-Weiss topological model $(\widehat{X},\widehat{\cb},\widehat{\mu},\widehat{T}=(\widehat{T}_i)_{i=1}^{k})$ associated to the given dynamical
system \linebreak $(X,\cb,\mu,\overline{T}=(T_i)_{i=1}^{k})$, and any $A_1\times A_2 \times \cdots A_k \in \widehat{\cb}^k$, put
\begin{eqnarray*}
\widehat{\nu}(A_1\times A_2 \times \cdots A_k)=
\frac13\sum_{n \in \Z}\frac1{2^{|n|}}\widehat{\mu}_{\Delta}(\widehat{T}^n(A_1 \times A_2 \times \cdots A_k)),
\end{eqnarray*}
where $\widehat{\mu}_{\Delta}$ is the diagonal measure on $\widehat{X}^k$ associated to $\widehat{\mu}$.\\
\noindent{}From this, we consider the non-singular dynamical system $(\widehat{X}^k,\bigotimes_{j=1}^{k}\widehat{\cb},\widehat{\lambda},\widehat{T})$, where
$$\widehat{\lambda}=\frac{\widehat{\nu}+\bigotimes_{i=1}^{k}\widehat{\mu}}{2}.$$ Hence, by the $L^2$-convergence combined with the same reasoning as in Lemma \ref{curcial}, for any
$A_1\times A_2 \times \cdots \times A_k \in \widehat{\ca}^k$, we have
$$\Big\|\frac{1}N\sum_{j=1}^{N}\1_{A_1\times A_2 \times \cdots  \times A_k}\circ \widehat{T}^n-\bigotimes_{j=1}^{k}\mu(A_1 \times A_2 \times \cdots \times A_k)\Big\|_{L^2(\widehat{\lambda})} \tend{N}{+\infty}0.$$
Therefore, for any $f_1 \otimes f_2 \otimes \cdots \otimes f_k \in \mathcal{C}(\widehat{X})^k$, we get
 $$\Big\|\frac1{N}\sum_{n=0}^{N-1}{\big(f_1\otimes f_2 \otimes \cdots \otimes f_k\big)}\circ \widehat{T}^n-
 \bigotimes_{j=1}^{k}\mu(f_1\otimes f_2 \otimes \cdots \otimes f_k)\Big\|_{L^2(\widehat{\lambda})}
\tend{N}{+\infty} 0.$$
This combined with the Stone-Weierstrass theorem gives that for any continuous functions $f$ on $\widehat{X}^k$, we have
\begin{eqnarray}\label{continueL2}
 \Big\|\frac1{N}\sum_{n=0}^{N-1}f\circ \widehat{T}^n-
 \bigotimes_{j=1}^{k}\mu(f)\Big\|_{L^2(\widehat{\lambda})}
\tend{N}{+\infty} 0.
\end{eqnarray}
Hence, by the standard argument, for any continuous functions $f$ on $\widehat{X}^k$, there is a subsequence $\mathcal{N}_f$ such that for $\widehat{\lambda}$-almost all $x \in \widehat{X}$, we have
$$\frac1{N}\sum_{n=0}^{N-1}f\circ \widehat{T}^n(x) \tend{N \in \mathcal{N}_f }{+\infty}\bigotimes_{j=1}^{k}\mu(f).$$
From this we have the following proposition.
\begin{Prop}\label{subsequence}Under the previous notation, there is a subsequence $\mathcal{N}$ and a Borel subset $D$ such that $\widehat{\lambda}(D)=1$, and for any continuous function $f$  on $\widehat{X}^k$, for any $x \in D$, we have
$$\frac1{N}\sum_{n=0}^{N-1}f\circ \widehat{T}^n
\tend{N \in \mathcal{N}}{+\infty}  \bigotimes_{j=1}^{k}\mu(f).$$
\end{Prop}
\begin{proof} Since $\mathcal{C}(\widehat{X})$ is separable we can choose a sequence $f_1,f_2,\cdots$ dense in
$\mathcal{C}(\widehat{X})$. Let $\mathcal{N}_1$ and $D_1$ such that, $\widehat{\lambda}(D_1)=1$ and for any $x \in D_1$, we have
$$\frac1{N}\sum_{n=0}^{N-1}f_1\circ \widehat{T}^n
\tend{N \in \mathcal{N}_1}{+\infty}  \bigotimes_{j=1}^{k}\mu(f_1).$$
Again, for $N \in \mathcal{N}_1$, there is a subsequence $\mathcal{N}_2 \subset \mathcal{N}_1$ and
$D_2 \subset D_1$ such that
 $\widehat{\lambda}(D_2)=1$ and for any $x \in D_2$, we have
$$\frac1{N}\sum_{n=0}^{N-1}f_2\circ \widehat{T}^n
\tend{N \in \mathcal{N}_2}{+\infty}  \bigotimes_{j=1}^{k}\mu(f_2).$$
Repeating this argument indefinitely we obtain sequences of integers
$\mathcal{N}_1 \supset \mathcal{N}_2 \supset \mathcal{N}_2 \supset \cdots$ where $\mathcal{N}_i=\{m_1^{(i)},
m_2^{(i)},\cdots\}$ and a sequence of Borel set $D_1 \supset D_2 \supset D_3 \supset \cdots$ with
$\widehat{\lambda}(D_i)=1$ such that, for any $j \leq i$ and $x \in D_i$; we have
$$\frac1{N}\sum_{n=0}^{N-1}f_j\circ \widehat{T}^n
\tend{N \in \mathcal{N}_i}{+\infty}  \bigotimes_{j=1}^{k}\mu(f_j).$$
Taking the diagonal sequence $\mathcal{N}$ and $D=\bigcap_{i} D_i$, it follows that for any $f \in \mathcal{C}(\widehat{X}^k)$, for any
$x \in D$,
$$\frac1{N}\sum_{n=0}^{N-1}f\circ \widehat{T}^n(x)
\tend{N \in \mathcal{N}}{+\infty}  \bigotimes_{j=1}^{k}\mu(f).$$
The proof of the proposition is complete.
\end{proof}
From this we are able to prove our main results. We start by proving Theorem \ref{main2}. For that,
we first remind the definition of the maximal ergodic inequality.
\begin{Def} Let $(Y,\cc,\sigma)$ be a Lebesgue space. We say that the Maximal Ergodic Inequality holds in $L^p(Y)$ for linear operator $T$, if setting
$$Sf(y)=\sup_{n \geq 1}\Big|\frac1{n}\sum_{j=0}^{n-1} (T^j(f))(y)\Big|,$$
we have
$$\gamma. \sigma\Big\{S(f)>\gamma\Big\} \leq ({\rm {constant}})\Big\|f\Big\|_p,$$
for all $f \in L^p(Y,\sigma)$ and $\gamma>0$.
\end{Def}

\begin{proof}[\textbf{Proof of Theorem \ref{main2}}] We claim that the maximal ergodic inequality does not hold for the non-singular dynamical system $(\widehat{X}^k,\widehat{\cb}^k,\widehat{\lambda},\widehat{T})$. Indeed, assume that the maximal ergodic inequality holds. Then, by the classical argument \cite[p.3]{Garcia} combined with the density of the subspace of continuous functions, it follows that for any Borel subset $A$ in $\widehat{X}^k$, the pointwise ergodic convergence along  the subsequence $\mathcal{N}$ holds, which is impossible since, by taking
$A=\bigcup_{n \in Z}\widehat{T}^n\Delta$, for any $N \in \N$, we have
$$\int \frac{1}{N}\sum_{n=1}^{N}\1_A(\widehat{T}^nx) d\widehat{\lambda}=\frac12 \neq 0=\bigotimes_{j=1}^{k}\mu(A).$$
This end the proof of the claim, and the proof of Theorem \ref{main2} is complete.
\end{proof}
\begin{rem}$~~~$
\begin{enumerate}[(i)]
\item The same argument yields that the maximal ergodic inequality does not hold for $\widehat{T}$ under $\widehat{\nu}$ and
 $\widehat{\mu}_{\Delta}$.
\end{enumerate}
\end{rem}
Furthermore, from Proposition \ref{subsequence}, we are able to deduce our main result concerning the pointwise convergence of multiple ergodic averages (P.C.M.E.A) along a subsequence holds (Theorem \ref{main1}). Indeed, we claim that if
$k \in \N^*$ and $(\widehat{X},\widehat{\cb},\widehat{\mu},\widehat{T}_i)_{i=1}^{k}$ is a  Stone-Jewett-Weiss topological model associated to the  finite family of dynamical systems
 $(X,\cb,\mu,T_i)_{i=1}^{k}$ where $\mu$ is a probability measure space, and  $\widehat{T}_1,T_2,\cdots,\widehat{T}_k$ are commuting weakly mixing transformations on $\widehat{X}$ such that for any $i \neq j$, the map $\widehat{T}_i \circ \widehat{T}_j^{-1}$ is ergodic. Then, there exists a subsequence $N_k$ such that, for
every $f_i \in L^{\infty}(\widehat{\mu})$, $i=1,\cdots,k$, the averages
$$ \frac1{N_k}\sum_{n=1}^{N_k} \prod_{i=1}^{k}f_i(\widehat{T}_i^nx)$$
converge almost everywhere to $\ds \prod_{j=1}^{k}\int f_j d\mu$.
\begin{proof}[\textbf{Proof of Theorem \ref{main1}}]We give the proof only for the case $k=2$; the other cases are left to the reader. Let $f_1,f_2 \in  L^{\infty}(\widehat{X})$. Then, by Lusin theorem \cite[p.53]{Rudin}, for any $\epsilon>0$, there exists $g_1,g_2 \in \mathcal{C}(\widehat{X})$ such that
\begin{eqnarray*}
||g_i||_{\infty} &\leq& ||f_i||_{\infty},\\
\widehat{\mu}\Big\{x~~:~~ f_i \neq g_i \Big\} &\leq& \epsilon, {\textrm{~~for~~each~~}}i=1,2.
\end{eqnarray*}
Therefore, by Proposition \ref{subsequence}, there exists a Borel set $X_{\epsilon}$ such that, for any $x \in X_{\epsilon}$, for a large $k$, we have
$$\Big|\frac1{N_k}\sum_{n=0}^{N_k-1}g_1(\widehat{T_1}^nx) g_2(\widehat{T_2}^nx)-\widehat{\mu}(g_1)\widehat{\mu}(g_2)\Big|
<\epsilon.$$
\noindent{}Now, write
\begin{eqnarray*}
&&\Big|\frac1{N_k}\sum_{n=0}^{N_k-1}f_1(\widehat{T_1}^nx) f_2(\widehat{T_2}^nx)-\widehat{\mu}(f_1)\widehat{\mu}(f_2)\Big|\\
&=&\Big|\frac1{N_k}\sum_{n=0}^{N_k-1}f_1(\widehat{T_1}^nx) f_2(\widehat{T_2}^nx)-g_1(\widehat{T_1}^nx) f_2(\widehat{T_2}^nx)+\\
&&\frac1{N_k}\sum_{n=0}^{N_k-1}g_1(\widehat{T_1}^nx) f_2(\widehat{T_2}^nx)-g_1(\widehat{T_1}^nx) g_2(\widehat{T_2}^nx)+\\
&&\frac1{N_k}\sum_{n=0}^{N_k-1}g_1(\widehat{T_1}^nx) g_2(\widehat{T_2}^nx)-\widehat{\mu}(g_1)\widehat{\mu}(g_2)
+\widehat{\mu}(g_1)\widehat{\mu}(g_2)-\widehat{\mu}(f_1)\widehat{\mu}(f_2)\Big|.
\end{eqnarray*}
Hence, by the triangle inequality, we get
\begin{eqnarray*}
&&\Big|\frac1{N_k}\sum_{n=0}^{N_k-1}f_1(\widehat{T_1}^nx) f_2(\widehat{T_2}^nx)-\widehat{\mu}(f_1)\widehat{\mu}(f_2)\Big|\\
&\leq&\Big|\frac1{N_k}\sum_{n=0}^{N_k-1}(f_1-g_1)(\widehat{T_1}^nx) f_2(\widehat{T_2}^nx)\Big|+
\Big|\frac1{N_k}\sum_{n=0}^{N_k-1}g_1(\widehat{T_1}^nx) (f_2-g_2)(\widehat{T_2}^nx)\Big|+\\
&&\Big|\frac1{N_k}\sum_{n=0}^{N_k-1}g_1(\widehat{T_1}^nx) g_2(\widehat{T_2}^nx)-\widehat{\mu}(g_1)\widehat{\mu}(g_2)\Big|
+\Big|\widehat{\mu}(g_1)\widehat{\mu}(g_2)-\widehat{\mu}(f_1)\widehat{\mu}(f_2)\Big|\\
&\leq&\frac1{N_k}\sum_{n=0}^{N_k-1}|f_1-g_1|(\widehat{T_1}^nx) \big\|f_2\big\|_{\infty}+
\frac1{N_k}\sum_{n=0}^{N_k-1} |f_2-g_2|(\widehat{T_2}^nx) \big\|g_1\big\|_{\infty}+\\
&&\Big|\frac1{N_k}\sum_{n=0}^{N_k-1}g_1(\widehat{T_1}^nx) g_2(\widehat{T_2}^nx)-\widehat{\mu}(g_1)\widehat{\mu}(g_2)\Big|
+\Big|\widehat{\mu}(g_1)\widehat{\mu}(g_2)-\widehat{\mu}(f_1)\widehat{\mu}(f_2)\Big|\\
&\leq&\frac1{N_k}\sum_{n=0}^{N_k-1}|f_1-g_1|(\widehat{T_1}^nx) \big\|f_2\big\|_{\infty}+
\frac1{N_k}\sum_{n=0}^{N_k-1} |f_2-g_2|(\widehat{T_2}^nx) \big\|f_1\big\|_{\infty}+\\
&&\Big|\frac1{N_k}\sum_{n=0}^{N_k-1}g_1(\widehat{T_1}^nx) g_2(\widehat{T_2}^nx)-\widehat{\mu}(g_1)\widehat{\mu}(g_2)\Big|
+\Big|\widehat{\mu}(g_1)\widehat{\mu}(g_2)-\widehat{\mu}(f_1)\widehat{\mu}(f_2)\Big|.
\end{eqnarray*}
By letting $k$ goes to infinity and applying Birkhoff ergodic theorem combined with Proposition \ref{subsequence}, we obtain
\begin{eqnarray*}
&&\Big|\frac1{N_k}\sum_{n=0}^{N_k-1}f_1(\widehat{T_1}^nx) f_2(\widehat{T_1}^nx)-\widehat{\mu}(f_1)\widehat{\mu}(f_2)\Big|
\\
&\leq& \int |f_1-g_1|d\widehat{\mu}\big\|f_2\big\|_{\infty}
+\int|f_2-g_2|d\widehat{\mu}\big\|f_1\big\|_{\infty}
+\Big|\widehat{\mu}(g_1)\widehat{\mu}(g_2)-\widehat{\mu}(f_1)\widehat{\mu}(f_2)\Big|.
\end{eqnarray*}
Similar arguments apply to the third term gives
$$\Big|\widehat{\mu}(g_1)\widehat{\mu}(g_2)-\widehat{\mu}(f_1)\widehat{\mu}(f_2)\Big| \leq
\widehat{\mu}(|g_1-f_1|)\big\|f_2\big\|_{\infty}+
\widehat{\mu}(|g_2-f_2|)\big\|f_1\big\|_{\infty}\leq
4\big\|f_2\big\|_{\infty}\big\|f_1\big\|_{\infty}\epsilon.$$
Summarizing, we have
$$\limsup\Big|\frac1{N_k}\sum_{n=0}^{N_k-1}f_1(\widehat{T_1}^nx) f_2(\widehat{T_1}^nx)-\widehat{\mu}(f_1)\widehat{\mu}(f_2)\Big|
\leq 8\big\|f_2\big\|_{\infty}\big\|f_1\big\|_{\infty}\epsilon.$$
We conclude by taking $\ds X'=\bigcap_{\overset{\epsilon >0}{\epsilon \in \Q}}X_{\epsilon}$ that, for every $x \in X'$,
$$\limsup\Big|\frac1{N_k}\sum_{n=0}^{N_k-1}f_1(\widehat{T_1}^nx) f_2(\widehat{T_1}^nx)-\widehat{\mu}(f_1)\widehat{\mu}(f_2)\Big|=0.$$
This finishes the proof of the theorem.
\end{proof}
\begin{rem}
 \begin{enumerate}[(i)]
 \item The proof above gives more, namely the pointwise convergence along a subsequence of the multiple ergodic averages holds for any $f_1,f_2,\cdots,f_k \in L^1(\widehat{\mu})$.\\
\item After the appearance of the second version of this paper, the author received an emails from Nikos Frantzikinakis and Pavel Zorin-Kranich in which they mentioned that the proof of the first main result of this paper (Theorem \ref{main1}) can be obtained without using the topological model arguments and by appealing to the maximal inequality.\\


The proof given here is based on the Birkhoff ergodic theorem. At this point, we should mention that by the Kolmogorov-Stein continuity principal theorem, the maximal inequality and the  Birkhoff ergodic theorem are equivalent. For the nice account, we refer to \cite{Garcia}. Nevertheless, in the proof of Theorem \ref{main1} we don't technically need to extend the maximal inequality to our setting. The same strategy was done recently in \cite{elabdal-al}. Therein, the authors reproved the Sarnak's result which say that for any dynamical system $(X,\ca, \mu, T)$ and for any function $f$ square-integrable,
$$\frac1{N}\sum_{j=1}^{N-1}\boldsymbol{\mu}(n)f(T^nx) \tend{N}{+\infty}0, \textrm{~~a.e.}$$
where $\boldsymbol{\mu}$ is the M\"{o}bius function given by
$$\boldsymbol{\mu}=\boldsymbol{\lambda}. \1_Q,$$
where $Q$ is the set of square-free integers. The integer $n$ is said to be square-free
if there is no prime number $p$ such that $n$ is in the class of $0$ mod
$p^2$.
\end{enumerate}
\end{rem}
At this point let us prove that we can apply Tuleca's theorem to the non-singular dynamical system
$(\widehat{X}^k,\widehat{\cb}^k,\widehat{\nu},\widehat{T})$. Indeed, we have the following lemma:
\begin{lem}\label{dissipative}The non-singular dynamical $(\widehat{X}^k,\widehat{\cb}^k,\widehat{\nu},\widehat{T})$
admits a $\sigma$-finite invariant measure $\overline{\nu}$ equivalent to
$\widehat{\nu}$.
\end{lem}
\begin{proof}By the aperiodicity of the maps $(T_iT_j^{-1})_{j=1, j\neq i}^{k}$, we have
$$T^n \Delta \cap T^m \Delta = \emptyset,$$
for any $n \neq m \in \Z$. Therefore, for any Borel set $A \subset \Delta$, let
$\overline{\nu}(A)=\widehat{\nu}(A).$ If $A \subset \widehat{T}^i\Delta$, for some $i$, then $\widehat{T}^{-i}A \subset \Delta$ and we put
$$\overline{\nu}(A)=\overline{\nu}(\widehat{T}^{-i}A)=\nu(\widehat{T}^{-i}A).$$
We thus have if $A \subset \widehat{T}^i\Delta$ then $\widehat{T}A \subset \widehat{T}^{i+1}\Delta$, and
$$\overline{\nu}(\widehat{T}A)=\widehat{\nu}(\widehat{T}^{-i-1}\widehat{T}A)=
\widehat{\nu}(\widehat{T}^{-i}A)=\overline{\nu}(A).$$
Finally, for any Borel set $A$, we put
$$\overline{\nu}(A)=\sum_{n \in \Z}\widehat{\nu}(A \cap \widehat{T}^i\Delta).$$
Therefore, it is a easy exercise to see that $\overline{\nu}$ is a $\sigma$-finite invariant measure. It is equivalent to
$\widehat{\nu}$ since $T$ is non-singular. The proof of the lemma is complete.
\end{proof}
\noindent{}The proof of Lemma \ref{dissipative} yields that the non-singular dynamical system $(\widehat{X}^k,\widehat{\cb}^k,\widehat{\nu},\widehat{T})$ is completely dissipative. Therefore, according to Hajian-Kakutani theorem \cite{H-kakutani}, there exists no equivalent finite $\widehat{T}$-invariant measure to $\widehat{\nu}$ on $\widehat{X}^k$. We summarize the previous result as follows
\begin{Prop}The non-singular dynamical system $(\widehat{X}^k,\widehat{\cb}^k,\widehat{\nu},\widehat{T})$ is completely dissipative with $\sigma$-invariant measure. We further have, for any $f \in L^1(\nu)$, for almost all $x \in \widehat{X}^k$,
$$\lim_{N \longrightarrow +\infty}\frac1{N}\sum_{j=0}^{N-1}f \circ \widehat{T}^n(x) \frac{d(\nu \circ \widehat{T}^n)}{d\nu}(x)$$
exists.
\end{Prop}

\section{Some remark on the problem of the almost sure convergence}
As in the previous section, let us consider $(\widehat{X},\widehat{\cb},\widehat{\mu},\widehat{T}=(\widehat{T}_i)_{i=1}^{k})$ the Stone-Jewett-Weiss topological model associated to $(X,\cb,\mu,\overline{T})$, and for any $A_1\times A_2 \times \cdots \times A_k \in \cb^k$, we still denoted by $\widehat{\nu}$ the measure given by
\begin{eqnarray*}
\widehat{\nu}(A_1\times A_2 \times \cdots A_k)=
\frac13\sum_{n \in \Z}\frac1{2^{|n|}}\widehat{\mu}_{\Delta}(\widehat{T}^n(A_1 \times A_2 \times \cdots \times A_k)),
\end{eqnarray*}
where $\widehat{\mu}_{\Delta}$ is the diagonal measure on $\widehat{X}^k$ associated to $\widehat{\mu}$.\\

In the same manner as before, by Lemma \ref{curcial},
for any  $A_1\times A_2 \times \cdots \times A_k \in \widehat{\ca}^k$, we have
$$\int \frac1{N}\sum_{n=0}^{N-1}\1_{A_1\times A_2 \times \cdots \times A_k}\circ \widehat{T}^n d\widehat{\lambda}
\tend{N}{+\infty}\bigotimes_{i=1}^{k}\widehat{\mu}\Big(A_1\times A_2 \times \cdots \times A_k\Big).$$
Therefore, for any $f_1 \otimes f_2 \otimes \cdots \otimes f_k \in \mathcal{C}(\widehat{X})^k$, we have
 $$\int \frac1{N}\sum_{n=0}^{N-1}\big({f_1\otimes f_2 \otimes \cdots \otimes f_k}\big)\circ \widehat{T}^n d\widehat{\lambda}
\tend{N}{+\infty} \int f_1\otimes f_2 \otimes \cdots \otimes f_k d(\bigotimes_{i=1}^{k}\widehat{\mu}).$$
This gives, by the standard argument, that  for any continuous functions $f$ on $\widehat{X}^k$,
\begin{eqnarray}\label{continue}
 \int  \frac1{N}\sum_{n=0}^{N-1}\widehat{T}^{n}f d\widehat{\lambda}
\tend{N}{+\infty} \int f\frac{d(\bigotimes_{i=1}^{k}\widehat{\mu})}{d\widehat{\lambda}}d\widehat{\lambda}.
\end{eqnarray}
Hence, for any continuous functions $f$ on $\widehat{X}^k$,

$$\int f \frac1{N} \sum_{n=1}^{N} {(\widehat{T})^*}^n\1 d\widehat{\lambda}
\tend{N}{+\infty} \int f  \frac{d(\bigotimes_{i=1}^{k}\widehat{\mu})}{d\widehat{\lambda}} d\widehat{\lambda}.$$

\noindent{}Consequently, for any Borel subset $A$ of $\widehat{X}^k$, the sequence
$\Big(\ds \int_A \frac1{N} \sum_{n=1}^{N} {(\widehat{T})^*}^n\1 d\widehat{\lambda}\Big)_{N \in \N}$ is a bounded sequence \footnote{Notice that for instance one may apply a standard argument, precisely, the portmanteau theorem \cite[p.15]{Bilingsly}  to get
$$\bigotimes_{j=1}^{k}\widehat{\mu}(\overset{\circ}{A}) \leq \liminf\Big(\frac1{N}\sum_{n=0}^{N-1}\widehat{\lambda}(\widehat{T}^n{A})\Big) \leq \limsup\Big(\frac1{N}\sum_{n=0}^{N-1}\widehat{\lambda}(\widehat{T}^n{A})\Big)\leq
2 \widehat{\lambda}(\overline{A}).$$
For the historical reference on this theorem, we refer the reader to the references in the Billingsly book.}.

Therefore, applying the Ryll-Nardzewski procedure (see section \ref{S2}), it follows that there is a $\widehat{T}$-invariant finitely additive measure $\rho$ given by
$$\rho(A)={\rm{\bf {MBlim}}}\Big(\int_{A} \frac1{N}\sum_{n=1}^{N} {(\widehat{T})^*}^n\1 d\widehat{\lambda}\Big).$$

We further have that for any $\widehat{T}$-invariant Borel subset $D$ of $\widehat{X}^k$,
 $\rho(D)=\widehat{\lambda}(D)$. $\rho$ is called a the Ryll-Nardzewski additive measure or the Ryll-Nardzewski charge. For a deep discussion of the connection between the Mazur-Banach limit and the invariant charge, we refer the reader to \cite{Sucheston2}, \cite{Goffman}. If the Hartman condition holds then the Ryll-Nardzewski additive measure $\rho$ is $\sigma$-additive, and following the idea of Ryll-Nardzewski, it suffices to apply the Birkhoff ergodic theorem to the dynamical system $(\widehat{X}^k,\bigotimes_{j=1}^{k}\widehat{\cb},\rho,\widehat{T})$. But, $\rho$ is only finitely additive probability measure on the $\sigma$-algebra $\bigotimes_{j=1}^{k}\widehat{\cb}$. Nevertheless, in our case, for any Borel set $A$ in the algebra, we have $\rho(A)=\bigotimes_{j=1}^{k}\widehat{\mu}(A)$. Therefore, for any continuous function $f \in \mathcal{C}(\widehat{X}^k)$, we have $\rho(f)=\bigotimes_{j=1}^{k}\mu(f)$ (see section 4). Indeed, for any continuous functions $f$ on $\widehat{X}^k$, there is a sequence of step functions $f_m=\sum_{i \in I_m}a_i \1_{A_i}$, where $I_m$ is finite set and $A_i$ a Borel set in the algebra, such that
 $f_m$ converge uniformly to $f$. Let $\epsilon>0$ and $m_0$ such that for any $m \geq m_0$, we have
 $$\|f_m-f\|_{\infty}<\epsilon.$$
Then, for any positive integer $N$, we can write
$$\frac1{N}\sum_{n=1}^{N}\widehat{\lambda}(f_m \circ \widehat{T}^n)-\epsilon
\leq  \frac1{N}\sum_{n=1}^{N}\widehat{\lambda}(f \circ \widehat{T}^n)
\leq \frac1{N}\sum_{n=1}^{N}\widehat{\lambda}(f_m \circ \widehat{T}^n)+\epsilon.$$
This combined with the classical properties of the Mazur-Banach limit gives
$$\rho(f_m)-\epsilon \leq \rho(f) \leq \rho(f_m)+\epsilon,$$
which implies that $\rho(f_m)$ converge to $\rho(f)$ and allows us to extend $\rho$ to the space of the continuous functions.
We further have, by the portmanteau theorem, that for any nonempty open set $O$ and any closet set $F$,
$$\rho(O) \geq \bigotimes_{j=1}^{k}\widehat{\mu}(O)>0,{\textrm{~~and~~}}
\rho(F) \leq \bigotimes_{j=1}^{k}\widehat{\mu}(F).$$
In the same manner we can define $\rho(f)$, for any lower semi-continuous (l.s.c) function $f$ and extend the previous classical inequalities to the case of the l.s.c functions and upper semi-continuous functions. Indeed, by  a classical argument, for any lower semi-continuous $f$ bounded from below there exists a sequence of lower semi-continuous $f_m$ which converge uniformly to $f$ such that $f_m$ is a finite sum of indicator functions of open set. Precisely, let us denote by $\mathcal{C}^{-1/2}(\widehat{X})$ (respectfully ${\mathcal{C}_u^{-1/2}}(\widehat{X})$) the set of l.s.c. functions on $X$ (respectfully $\mathcal{C}^{-1/2}(\widehat{X})$ equipped with uniform convergence topology) and without loss of generality we can assume that  $0 < f(x) \leq 1$ for all $x \in \widehat{X} $. Fix $m \geq 1$ and define $f_m$ by
$$f_m=\sum_{j=1}^{m}\frac{1}{m} \1_{O_j} {\textrm{~,~where~~}}O_j=\big\{x~~:~~f(x) >\frac{j}{m}\big\}. $$
Obviously, $O_j$ is a open set and it is easy to see that the indicator function of any open set is in $\mathcal{C}^{-1/2}(\widehat{X})$. We further have that ${\mathcal{C}_u^{-1/2}}(\widehat{X})$ is a closed positive lattice cone that contains the spermium of each of its subsets. Moreover, in the previous procedure, the function $f$ is truncated to $\frac{j}{m}$ if $f(x) \in [\frac{j}{m},\frac{j+1}{m})$. Thus $f_m \leq f$ and $\ds \sup_{x \in X}|f_m(x)-f(x)|\leq$$\frac{1}{k}$. The result follows by letting $m$ goes to infinity.
Therefore, $\rho(f_m)$ converge to $\rho(f)$, by the classical property of the Banach limit. We thus get, again by the portmanteau theorem, that
\begin{eqnarray}\label{ilsc}
\rho(f) \geq \bigotimes_{j=1}^{k}\widehat{\mu}(f), \textrm{~~for~~every ~~l.~s.~c.~} f
\textrm{~~bounded~~from~~below,}
\end{eqnarray}
and
\begin{eqnarray}\label{ulsc}
\rho(f) \leq \bigotimes_{j=1}^{k}\widehat{\mu}(f), \textrm{~~for~~every ~~u.~~s.c.~~} f
\textrm{~~bounded~~from~~above.}
\end{eqnarray}

Notice that under our assumption $\widehat{X}$ is compact and we don't need to assume that the l.s.c functions (resp. the u.s.c functions) are bounded below (resp. bounded above). With this in mind, one may try to apply the standard argument from ergodic theory to get some kind of ergodic theorem for the charge $\rho$ but it is turn out that $\rho$ can not be $\sigma$-additive (see Remark below).

\begin{rem}$~~~$
\begin{enumerate}[(a)]
\item Notice that the proof above gives more, namely the Hartman condition does not hold for the
 non-singular dynamical system $(X,\cb,\widehat{T},\widehat{\lambda})$. If not, we get that $\rho$ is $\sigma$-additive on the $\sigma$-algebra (notice that it is obvious that $\rho$ is $\sigma$-additive on the algebra \footnote{See for example the proof of Lemma 4.6 in \cite{Gklauss}.}). Indeed, let $A_i$ be a disjoint Borel set and $l$ be a positive integer. Then
    $$\rho\Big(\bigcup_{i \geq l}A_i\Big) \leq K \widehat{\lambda}\Big(\bigcup_{i \geq l}A_i\Big) \tend{l}{+\infty}0.$$
 Hence,
     $$\rho\Big(\bigcup_{i \geq 1}A_i\Big)=\sum_{i=1}^{l}\rho(A_i)+\rho\Big(\bigcup_{i \geq l}A_i\Big)\tend{l}{+\infty}
     \sum_{i=1}^{+\infty}\rho(A_i).$$
 We further have, from our proof, that for any continuous function $f$, $\rho(f)=\bigotimes_{j=1}^{k}\mu(f)$,  and $\rho(D)=\widehat{\lambda}(D)$ for any $\widehat{T}$-invariant set. This gives that
 $$\bigotimes_{j=1}^{k}\mu(\bigcup_{n \in Z}\widehat{T}^n\Delta)=\widehat{\lambda}(\bigcup_{n \in Z}\widehat{T}^n\Delta),$$ which is impossible. We thus get, by Theorem 1 in \cite{Ryll}, that the pointwise ergodic theorem does not hold for the
 non-singular dynamical systems $(X,\cb,\widehat{T},\widehat{\lambda})$ and $(X,\cb,\widehat{T},\widehat{\nu})$
\item By construction, the charge $\rho$ satisfy: for any  $B$ in the algebra, $\rho(B)>0$ and for any $n\in \Z$, we have $\rho(T^nB \bigcap B)>0$.
\item Let $(Y,\cb,\rho,S)$ be a dynamical system with finitely additive probability measure $\rho$ and $\cb$ be a $\sigma$-algebra generating by some dense algebra $\ca$. We suppose that the restriction of $\rho$ coincide on the $\sigma$-algebra of $S$-invariant set with some $\sigma$-additive measure $\lambda$ (not necessary invariant under $S$). We further suppose that $\lambda \geq K.\mu$ and $\mu$ is a invariant measure.  Then, for any Borel set $D$ such that $S^nD \bigcap S^mD =\emptyset$, for any $n \neq m$, we have
    $\rho(\bigcup_{n \in \Z}S^nD)=1$. Indeed, $\rho(\bigcup_{n \in \Z}S^nD)=\lambda(\bigcup_{n \in \Z}S^nD)=1$.
\item Note that we have actually proved that there is a dynamical system with a finitely additive probability measure for which the the Poincar\'e recurrence Theorem holds.
\item Our proof yields also that there is non-singular maps such that
$\bigotimes_{j=1}^{k}\mu$ is in the convex set $\Pi(\widehat{T},\widehat{\lambda})$ given by
$$\Pi(\widehat{T},\widehat{\lambda})=\overline{{\rm{env}}\{\widehat{\lambda} \circ \widehat{T^n}\}}^{W^*},$$
that is, the convex set
\begin{eqnarray*}
\Pi^*(\widehat{T},\widehat{\lambda})&=&\overline{\Big\{\rho \in B(0,1]~~:~~ \rho \in \Pi(\widehat{T},\widehat{\lambda}) {~~and~~} \rho \circ \widehat{T}=\rho,~~\rho << \widehat{\lambda}\Big\}}^{W^*}\\
&\subset& B(0,1].
\end{eqnarray*}
is not empty, $B(0,1]$ is the unit Ball in the dual space of $C(\widehat{X})$. We thus get that it is weakly compact by Alaoglu-Banach-Bourbaki. This combined with the Krein-Milman theorem yields that it is the weakly star closure of the convex hall of its extremal points $E(\widehat{T})$, that is,
$$\Pi_1(\widehat{T},\widehat{\lambda})=\overline{{\rm{env}}E(\widehat{T})}.$$ Notice that since $\widehat{X}^k$ is a compact metric space, it follows that $E(\widehat{T})$ is the set of all $\widehat{\rho} \in E(\widehat{T})$ for which $\widehat{T}$ is an ergodic measure-preserving transformation of $(X^k, \bigotimes_{j=1}^{k}B,\widehat{\rho})$.
\item It follows from $(e)$ that one can substitute $\widehat{\lambda}$ by any limit of convex
combinations of the ergodic joining of $\widehat{T}_1,\widehat{T}_2,\cdots,\widehat{T}_k$. Here, the chosen ergodic joining is
$\bigotimes_{j=1}^{k}\mu$ and the off-diagonal measure $\mu_{\widehat{T}^n\Delta}$.
\item Following the Ryll-Nardzewski procedure, one may construct
 a large class of invariant Ryll-Nardzewski probability measures $\widehat{\rho} \leq K.\widehat{\lambda}$ associated to the subclass of compact subset $K$ such that $\widehat{\lambda}(K)>0$. Indeed, the sequence
$\frac1{N}\sum_{n=1}^{N}\widehat{\lambda}_K(f\circ \widehat{T}^n)$ is a bounded sequence, where
$\widehat{\lambda}_K$ is the restriction of $\widehat{\lambda}$ to $K$ normalized. Therefore, by Mazur-Banach limit procedure, we define
$$\rho(f)={\rm{\bf {MBlim}}}\Big(\frac1{N}\sum_{n=0}^{N-1}\widehat{\lambda}(f\circ \widehat{T}^n)\Big).$$
It follows that $\rho$ is a bounded operator on the space of continuous function $\mathcal{C}(\widehat{X})$. We further have that
$\rho \circ \widehat{T}=\rho$ and if $f$ is a $\widehat{T}$-invariant function then $\rho(f)=\widehat{\lambda}(f)$.
Hence, by Riesz representation theorem, there is a unique
Borel probability measure $\widehat{\rho}$ such that $\rho(f)=\widehat{\rho}(f)$ for all $f \in \mathcal{C}(\widehat{X})$.
The condition $(1)$ implies that, for any Borel set $A$ in the $\sigma$-Borel algebra $\cb(\widehat{X})$,  we have
$$\widehat{\rho}(A) \leq K. \widehat{\lambda}(A).$$
We called such class a Ryll-Nardzewski class of $\widehat{T}$-invariant probability measures for a given non-singular map.
\end{enumerate}
\end{rem}
\begin{thank}I would like to express my warmest thanks to Professor XiangDong Ye for his invitation to Hefei University of Technology and for bringing to my attention the problem of the pointwise convergence of Furstenberg average and for the e-correspondences. Thanks also to Professor Bernard Host, Professor Nikos Frantzikinakis and Professor Idriss Assani for their interest to this work and their comments.
The author is also thankful to his colleague for pointing out a mistake in the first draft posted on arXiv. Finally, my posthumous thanks to Professor Antoine Brunel who introduce me to the work of R. Sato.
\end{thank}

\end{document}